\documentclass[draft]{amsart}

\usepackage{amssymb}
\usepackage{color} 

\newcommand{\eql}{\kern-1ex &=& \kern-1ex}
\newcommand{\N}{\mathbb{N}}

\newcommand{\R}{\mathbb{R}}
\newcommand{\Ker}{\mathop\mathrm{Ker}\nolimits}
\newcommand{\coKer}{\mathop\mathrm{coKer}\nolimits}

\newcommand{\ind}{\mathop\mathrm{ind}\nolimits}

\newcommand\multimapsto{
\hbox{\begin{picture}(16,0)
\kern 0.6ex\put(0.5,0.6){\line(0,1){3.6}}$\multimap$
\end{picture}}}

\newcommand{\pb}{\kern -2ex}
\renewcommand{\Im}{\mathop\mathrm{Img}\nolimits}
\renewcommand{\d}{\delta}
\renewcommand{\l}{\lambda}
\newcommand{\e}{\varepsilon}
\newcommand{\s}{\sigma}
\renewcommand{\a}{\alpha}

\newcommand{\g}{\gamma}

\renewcommand{\t}{\theta}

\theoremstyle{plain}
\newtheorem{theorem}{Theorem}[section]
\newtheorem{corollary}[theorem]{Corollary}
\newtheorem{lemma}[theorem]{Lemma}
\newtheorem{proposition}[theorem]{Proposition}
\newtheorem{conjecture}[theorem]{Conjecture}
\newtheorem{assertion}[theorem]{Assertion}
\newtheorem{remark}[theorem]{Remark}
\theoremstyle{definition}
\newtheorem{definition}[theorem]{Definition}
\newtheorem{example}[theorem]{Example}

\numberwithin{equation}{section}

\begin{document}

\title[Global persistence of eigenvectors]{Global persistence of the unit eigenvectors of perturbed eigenvalue problems in Hilbert spaces}
\author[P.\ Benevieri]{Pierluigi Benevieri}
\author[A.\ Calamai]{Alessandro Calamai}
\author[M.\ Furi]{Massimo Furi}
\author[M.P.\ Pera]{Maria Patrizia Pera}

\thanks
{The first, second and fourth authors are members of the Gruppo Nazionale per l'Analisi Matematica, la Probabilit\`a e le loro Applicazioni (GNAMPA) of the Istituto Nazionale di Alta Matematica (INdAM)}
\thanks{A.\ Calamai is partially supported by GNAMPA - INdAM (Italy)}


\address{Pierluigi Benevieri -
Instituto de Matem\'atica e Estat\'istica,
Universidade de S\~ao Paulo,
Rua do Mat\~ao 1010,
S\~ao Paulo - SP - Brasil - CEP 05508-090 -
 {\it E-mail address: \tt
pluigi@ime.usp.br}}
\address{Alessandro Calamai -
Dipartimento di Ingegneria Civile, Edile e Architettura,
Universit\`a Politecnica delle Marche,
Via Brecce Bianche,
I-60131 Ancona, Italy -
 {\it E-mail address: \tt
calamai@dipmat.univpm.it}}
\address{Massimo Furi - Dipartimento di Matematica e Informatica ``Ulisse Dini'', 
Universit\`a degli Studi di Firenze,
Via S.\ Marta 3, I-50139 Florence, Italy -
 {\it E-mail address: \tt
massimo.furi@unifi.it}}
\address{Maria Patrizia Pera - Dipartimento di Matematica e Informatica ``Ulisse Dini'',
Universit\`a degli Studi di Firenze,
Via S.\ Marta 3, I-50139 Florence, Italy -
 {\it E-mail address: \tt
mpatrizia.pera@unifi.it}}

\begin{abstract}
We consider the nonlinear eigenvalue problem $Lx + \e N(x) = \l Cx$, $\|x\|=1$, where $\e,\l$ are real parameters, $L, C\colon G \to H$ are bounded linear operators between separable real Hilbert spaces, and $N\colon S \to H$ is a continuous map defined on the unit sphere of $G$.
We prove a global persistence result regarding the set $\Sigma$ of the
\emph{solutions} $(x,\e,\l) \in S \times \R\times \R$ of this problem.
Namely, if the operators $N$ and $C$ are compact, under suitable assumptions on a solution $p_*=(x_*,0,\l_*)$ of the unperturbed problem, we prove that the connected component of $\Sigma$ containing $p_*$ is either unbounded or meets a triple $p\sp*=(x\sp*,0,\l\sp*)$ with $p\sp* \not= p_*$.
When $C$ is the identity and $G=H$ is finite dimensional, the assumptions on $(x_*,0,\l_*)$ mean that $x_*$ is an eigenvector of $L$ whose corresponding eigenvalue $\l_*$ is simple.
Therefore, we extend a previous result obtained by the authors in the finite dimensional setting.

Our work is inspired by a paper of R.\ Chiappinelli concerning the
local persistence property of the unit eigenvectors of perturbed self-adjoint operators in a real Hilbert space.
\end{abstract}



\maketitle


\centerline{\textit{Dedicated to the memory of our dear friend and exceptional mathematician Alfonso Vignoli}}

\section{Introduction}
\label{Introduction}

In this paper we study a \textit{nonlinear eigenvalue problem} of the type
\begin{equation}
\label{problem-intro}
\begin{cases}
\;Lx + \e N(x) = \l Cx,\\[.3ex]
\;x \in S,
\end{cases}
\end{equation}
where $\e, \l \in \R$, $L,C\colon G \to H$ are bounded linear operators between real Hilbert spaces, and $N\colon S \to H$ is a continuous map defined on the unit sphere of $G$.

Problem \eqref{problem-intro} can be thought as a nonlinear perturbation of the eigenvalue problem $Lx = \l Cx$, where, as usual, by abuse of terminology, $\l_* \in \R$ is said to be an \emph{eigenvalue of (the equation) $Lx = \l Cx$} if the operator $L-\l_* C$ is not injective.

By a \emph{solution} of \eqref{problem-intro} we mean a \emph{triple} $(x,\e,\l)$ which satisfies the system, and we call $(\e,\l)$ the \emph{eigenpair} corresponding to the \emph{(unit) eigenvector} $x$.
The solutions and the eigenpairs with $\e=0$ are said to be \emph{trivial}.
The set of all the solutions is denoted by $\Sigma$, while $\Sigma_0$ stands for its subset of the trivial ones.
Obviously, $\Sigma$ and $\Sigma_0$ are closed in $S \times \R \times \R$.

\medskip
Recently, in \cite{BeCaFuPe-s3}, under the assumptions that $G = H = \R\sp{n}$ and that $C$ is the identity $I$, we obtained a sort of ``global persistence'' of the solution triples of the above problem (see Corollary \ref{continuation: finite dimension} below). That is, we proved that if $\l_* \in \R$ is a simple eigenvalue of $L$ (in the usual sense) and $x_*$ is any one of the two corresponding unit eigenvectors, then one gets the following
\begin{assertion}
\label{assertion}
The set $\Sigma\setminus\Sigma_0$ of the nontrivial solutions of \eqref{problem-intro} has a connected subset whose closure contains $p_*=(x_*,0,\l_*)$ and is either unbounded or meets a trivial solution $p\sp*=(x\sp*,0,\l\sp*)$ different from $p_*$.
\end{assertion}

Observe that this assertion does not imply that $\l\sp*$ is different from $\l_*$. However, if $\l\sp* = \l_*$, one necessarily has $p\sp*=(-x_*,0,\l_*)$.

\smallskip
Taking into account that the closure of a connected set is connected, from Assertion \ref{assertion} one gets that \emph{the component of $\Sigma$ containing $p_*$ is either unbounded or meets a trivial solution $p\sp* \not= p_*$.}

We point out that, given any trivial solution $p_*=(x_*,0,\l_*)$ of \eqref{problem-intro}, this last statement is meaningful when (and only when) the kernel of $L-\l_*C$ is one dimensional, due to the fact that (only) in this case the sphere of the unit eigenvectors corresponding to $\l_*$ is disconnected.

\medskip
In this paper we extend the global persistence result obtained in \cite{BeCaFuPe-s3} to the infinite dimensional setting (see Theorem \ref{continuation} below). Namely, given a trivial solution $p_*=(x_*,0,\l_*)$ of \eqref{problem-intro}, we get Assertion \ref{assertion} under the following assumptions:
\begin{itemize}
\item
$\Ker (L-\l_*C) = \R x_*$,
\item
$Cx_*\not= 0$,
\item
$\Im (L-\l_*C) \oplus C (\Ker (L-\l_*C)) = H$,
\item
$G$ and $H$ are separable,
\item
$N$ and $C$ are compact.
\end{itemize}

For the sake of simplicity, when a trivial solution $(x_*,0,\l_*)$ of \eqref{problem-intro} satisfies the first three of the above assumptions, we say that it is a \emph{simple solution}.

One can easily check that, given a linear operator $L\colon \R\sp{n} \to \R\sp{n}$, then $\l_* \in \R$ is a simple eigenvalue of $L$ with corresponding unit eigenvector $x_*$ if and only if the triple $(x_*,0,\l_*)$ is a simple solution of \eqref{problem-intro} in which $G=H=\R\sp{n}$ and $C=I$.

\medskip
We do not know whether or not the above assumptions imply that the eigenvalues $ \l_*$ and $ \l\sp*$ in Assertion \ref{assertion} are different. 
Nevertheless, in \cite{BeCaFuPe-s2}, we tackled the problem of the global persistence of the eigenvalues (more precisely, of the eigenpairs) of \eqref{problem-intro} and, in particular, by means of the Leray--Schauder degree theory we obtained the following

\begin{theorem}[Global continuation of eigenpairs]
\label{continuation of eigenpairs intro}
Regarding problem \eqref{problem-intro}, assume that the operator $L$ is Fredholm of index zero, that $C$ and $N$ are compact, and that, for some $\l_* \in \R$, the kernel of $L -\l_*C$ is odd dimensional and the condition
\begin{equation*}
\Im (L-\l_*C) + C (\Ker (L-\l_*C)) = H
\end{equation*}
is satisfied.

Then, in the set of all the eigenpairs $(\e,\l)$ of \eqref{problem-intro}, the connected component containing $(0,\l_*)$ is either unbounded or meets a trivial eigenpair $(0,\l\sp*)$ with \hbox{$\l\sp* \not= \l_*$}.
\end{theorem}

Because of Theorem \ref{continuation of eigenpairs intro}, we are inclined to believe that our main result (Theorem \ref{continuation}) could be sharpened according to the following conjecture that until now we have not been able to prove or deny:

\begin{conjecture} 
\label{conjecture-intro}
Let $(x_*,0,\l_*)$ be a simple solution of problem \eqref{problem-intro}. Suppose that $G$ and $H$ are separable, and that $N$ and $C$ are compact.

Then, the set of the nontrivial solutions of \eqref{problem-intro} has a connected subset whose closure contains $(x_*,0,\l_*)$ and is either unbounded or meets a trivial solution $(x\sp*,0,\l\sp*)$ with $\l\sp* \not= \l_*$.

\end{conjecture} 

Our investigation is mainly inspired by a paper of R.~Chiappinelli \cite{Chi2003}, who obtained a ``local persistence'' result for problem \eqref{problem-intro}.
Namely, assuming that
\begin{itemize}
\item
$L$ is a \emph{self-adjoint} operator defined on $G$,
\item
$C = I$ is the identity of $G$,
\item
$N \colon S \to G$ is Lipschitz continuous,
\item
$\l_* \in \R$ is an isolated simple eigenvalue of $L$,
\item
$x_*$ is any one of the two unit eigenvectors corresponding to $\l_*$,
\end{itemize}
he proved that,
\emph{defined in a neighborhood $V$ of $0 \in \R$, there exists a $G$-valued Lipschitz curve $\e \mapsto x_\e$, as well as a real Lipschitz function $\e \mapsto \l_\e$, such that
\[
Lx_\e + \e N(x_\e) = \l_\e x_\e, \quad
\|x_\e\|=1,
\]
for any $\e \in V$. Moreover, when $\e = 0$ one has $x_0=x_*$ and $\l_0=\l_*$.}

\medskip
The hypotheses of our Theorem \ref{continuation} seem incompatible with the assumptions of Chiappinelli's local persistence result, since, what in \eqref{problem-intro} is the compact operator $C$, in Chiappinelli's case is the identity $I$, which is not compact when the space is infinite dimensional. 
Nevertheless, under some natural conditions on the operator $L$, our result applies. This is the case, for example, when $L$ is compact or, more generally, when it is of the type $\l_* I + C$, with $\l_* \in \R$ and $C$ compact.
To see this, put $\e= -\s/\mu$ and $\l=\l_*+1/\mu$, and observe that the equation $Lx+\e N(x) = \l x$ becomes $x + \s N(x) = \mu Cx$, which is as in \eqref{problem-intro} with the identity in place of $L$.

\medskip
Further results regarding the local persistence of eigenvalues, as well as unit eigenvectors, have been obtained in \cite{BeCaFuPe-s1, Chi2017, ChiFuPe1, ChiFuPe2, ChiFuPe3, ChiFuPe4} in the case in which the eigenvalue $\l_*$ is not necessarily simple.
For a general review on nonlinear eigenvalue problems and applications to differential equations, see e.g.\ 
\cite{Chi2018} and references therein.

The proof of Theorem \ref{continuation} does not need any advanced tool (such as Leray-Schauder degree theory) and mainly requires basic concepts in Differential Topology and Functional Analysis that can be found in textbooks such as \cite{Lang, Milnor, Taylor & Lay}.

A crucial result in our investigation is Lemma \ref{diffeomorphism}, which concerns the unperturbed eigenvalue problem: it states that \emph{if $(x_*,0,\l_*)$ is a simple solution of \eqref{problem-intro}, then the map $\Psi\colon S \times \R \to H$, given by $(x,\l) \mapsto Lx-\l Cx$, establishes a diffeomorphism from a neighborhood of $(x_*,\l_*)$ in $S \times \R$ onto a neighborhood of $0 \in H$}.

We close the paper with some illustrating examples showing situations in which Theorem \ref{continuation} applies. We also show that, in our main result, the hypothesis that the ``starting'' trivial solution is simple cannot be removed.

\section{Notation and preliminaries}
\label{Preliminaries}

In this section, in addition to introducing our notation, we will recall some topological and algebraic concepts that will be needed in the following (for general reference see e.g.~\cite{Lang, Milnor, Taylor & Lay}).

\smallskip
Throughout the paper, $G$, $H$ and $G\times\R\times\R$ indicate real Hilbert spaces.
In each one of these spaces, the inner product is denoted by $\langle\cdot,\cdot\rangle$, or by a similar symbol, such as $\langle\cdot,\cdot\rangle'$, only when a possible misunderstanding regarding the hosting space may occur.
For example, the inner product of two elements $p_1 = (x_1,\e_1,\l_1)$ and $p_2 = (x_2,\e_2,\l_2)$ of $G\times\R\times\R$ is defined in the most natural way as follows:
\[
\langle p_1, p_2\rangle = \langle x_1, x_2\rangle' + \e_1\e_2 + \l_1\l_2,
\]
where, here, $\langle x_1, x_2\rangle'$ is the inner product of $x_1, x_2 \in G$.

The norm in any Hilbert space will be tacitly assumed to be the standard one associated with the inner product; namely, $\|\cdot\| = \sqrt{\langle\cdot,\cdot\rangle}$.

By a manifold we shall always mean a smooth (i.e.\ $C\sp\infty$) boundaryless submanifold of a real Hilbert space.
Given a manifold $\mathcal M$ in a Hilbert space, say $G$, and given a point $p \in \mathcal M$, the tangent space $T_p(\mathcal M)$ of $\mathcal M$ at $p$ will be always identified with a vector subspace of $G$; so that any $v \in T_p(\mathcal M)$ is the derivative $\g'(0)$ of a smooth path $\g\colon (-1,1) \to G$ whose image is in $\mathcal M$ and such that $\g(0) = p$.

Obviously, if $\mathcal M$ and $\mathcal N$ are two manifolds such that $\mathcal M \subseteq \mathcal N$ and $p \in \mathcal M$, then $T_p(\mathcal M)$ is a subspace of $T_p(\mathcal N)$.

By a smooth (i.e.\ $C\sp\infty$) map $f\colon \mathcal M \to \mathcal N$ between two manifolds $\mathcal M \subseteq G$ and $\mathcal N \subseteq H$ we mean the restriction (to $\mathcal M$ as domain and to $\mathcal N$ as codomain) of a smooth map $\hat f\colon U \to H$ defined on an open neighborhood $U$ of $\mathcal M$ in $G$.

Throughout this section, $\mathcal M$ and $\mathcal N$ denote two manifolds embedded in $G$ and $H$ respectively, and $f$ is a smooth map from $\mathcal M$ into $\mathcal N$.

Given $p \in \mathcal M$, the differential $df_p\colon T_p(\mathcal M) \to T_{f(p)}(\mathcal N)$ is the restriction of the Fr\'echet differential $d\hat f_p \colon G \to H$ of any smooth extension $\hat f\colon U \to H$ of $f$ to an open neighborhood $U$ of $\mathcal M$. To check that $df_p$ is well defined, think about the tangent vectors as derivatives of smooth curves.

If $f$ is bijective and its inverse $f\sp{-1}\colon \mathcal N \to \mathcal M$ is smooth, then $f$ is said to be a \emph{diffeomorphism (of $\mathcal M$ onto $\mathcal N$)}.
One also says that $f$ \emph{maps $\mathcal M$ diffeomorphically onto $\mathcal N$}.
It is known that $f$ is a diffeomorphism if and only if it is a homeomorphism with $df_p$ invertible for any $p \in \mathcal M$.
Notice that the function $t \mapsto t\sp3$ from $\R$ onto itself is smooth, is a homeomorphism, but not a diffeomorphism.

Here we will distinguish between differential and derivative.
The first one is always a linear map; the second one, when it makes sense, is a ``representative'' of the differential.
For example, if $f\colon \R\sp{k} \to \R\sp{s}$ and $p \in \R\sp{k}$, the derivative of $f$ at $p$, $f'(p)$, is the matrix representing the differential $df_p$ (the Jacobian matrix, in this case); if $f$ is a curve into $G$ and $p$ is in the domain of $f$, then $f'(p)$ is a vector of $G$, namely $f'(p)=df_p(1)$.

\smallskip
The following well-known result regarding $1$-dimensional manifolds will be crucial in the next section.

\begin{theorem}[On the classification of $1$-dimensional manifolds]
\label{one-dimensional manifolds}
Any connected (boundaryless) $1$-dimensional manifold is either diffeomorphic to an open real interval or to a circle.
\end{theorem}

The following result concerning diffeomorphisms, known as the ``Local Inverse Function Theorem'', will be useful in the next section.
\begin{theorem}[On the local diffeomorphism]
\label{local diffeomorphism}
Let $f\colon \mathcal M \to \mathcal N$ be a smooth map between two manifolds and let $p \in \mathcal M$.

Then $f$ maps diffeomorphically an open neighborhood of $p$ in $\mathcal M$ onto an open neighborhood of $f(p)$ in $\mathcal N$ if and only if the differential $df_p\colon T_p(\mathcal M) \to T_{f(p)}(\mathcal N)$ is invertible.
\end{theorem}

Let, as above, $f\colon \mathcal M \to \mathcal N$ be a smooth map between two manifolds. Recall that $p \in \mathcal M$ is said to be a \emph{regular point (of $f$)} if the differential $df_p$ is surjective; otherwise $p$ is called a \emph{critical point}.
An element $q \in \mathcal N$ is a \emph{critical value (of $f$)} if $f\sp{-1}(q)$ does not contain critical points; otherwise $q$ is said to be a \emph{regular value}.
Obviously, if $q$ is not in the image of $f$, then it is a regular value.
Notice that ``points'' are in the domain and ``values'' in the codomain.

\smallskip
An useful consequence of the Implicit Function Theorem is the following
\begin{theorem}[On regularly defined finite codimensional manifolds]
\label{regularly defined finite codimensional manifolds}
Let $f$ be a smooth map from $\mathcal M$ into $\mathcal N$.
Assume that $\mathcal N$ is finite dimensional and that $q \in \mathcal N$ is a regular value for~$f$.

Then $f\sp{-1}(q)$, if nonempty, is a manifold whose codimension in $\mathcal M$ is the same as the dimension of $\mathcal N$. Moreover, for any $p \in f\sp{-1}(q)$, one has $T_p(f\sp{-1}(q)) = \Ker df_p$.
\end{theorem}
For example, the unit sphere $S$ of $G$ is a $1$-codimensional submanifold of $G$. Moreover, given $x_* \in S$, one has $T_{x_*}(S) = (\R x_*)\sp\perp = \{\dot x \in G: \langle x_*, \dot x\rangle = 0$\}. 
To see this, define $f \colon G \to \R$ by $f(x) = \|x\|\sp2$, consider the regular value $q = 1 \in \R$, and apply Theorem \ref{regularly defined finite codimensional manifolds}.

\medskip
Let $X$ and $Y$ denote metric spaces, and let $g\colon X \to Y$ be a continuous map.
One says that $g$ is \emph{compact} if it sends bounded subsets of $X$ into relatively compact subsets of $Y$, and that it is \emph{locally compact} if, given any $x \in X$, the restriction of $g$ to a convenient neighborhood of $x$ is a compact map.

The map $g$ is said to be \emph{proper} if $g\sp{-1}(K)$ is compact for any compact subset $K$ of $Y$.
As one can easily check, if $g$ is proper, then is a \emph{closed map}, in the sense that $g(D)$ is closed in $Y$ whenever $D$ is closed in $X$.

Recall that a subset $D$ of $X$ is called \emph{locally compact} if any point of $D$ admits a neighborhood (in $D$) which is compact.
Clearly, any (relatively) closed or (relatively) open subset of a locally compact set is locally compact.
Observe that the union of two locally compact sets could not be locally compact (for example, add a point to the boundary of the open unit disk in $\mathbb C$).

\medskip
Let hereafter $E$ and $F$ denote real Banach spaces, and let $L(E,F)$ be the Banach space of all the bounded linear operators from $E$ into $F$.

\smallskip
Recall that $L \in L(E,F)$ is said to be a \emph{Fredholm operator} if both its kernel, $\Ker L$, and its co-kernel, $\coKer L = F/\Im L$, are finite dimensional. The difference of these dimensions is called the \emph{index} of $L$ and denoted by $\ind L$.

The following are some useful properties regarding Fredholm operators:

\begin{itemize}
\item
[(1)] \emph{the image of a Fredholm operator $L \in L(E,F)$ is a closed subspace of its codomain $F$;}

\item
[(2)] \emph{the index of the Fredholm operators is stable, in the sense that, in $L(E,F)$, the set of Fredholm operators of a given index is open;}

\item
[(3)] \emph{if $L \in L(E,F)$ is Fredholm and $C \in L(E,F)$ is compact, then $L+C$ is Fredholm with $\ind(L+C) = \ind L$;}

\item
[(4)] \emph{the composition of two Fredholm operators is still Fredholm, and its index is the sum of the indices of the composite operators;}

\item
[(5)] \emph{Fredholm operators are proper on bounded closed subsets of their domains.}
\end{itemize}

\medskip
A smooth map $f\colon \mathcal M \to \mathcal N$ is said to be \emph{Fredholm of index $n$} if so is the linear operator $df_p\colon T_p(\mathcal M) \to T_{f(p)}(\mathcal N)$ for any $p \in \mathcal M$.

\smallskip
A fundamental result regarding Fredholm maps is the following infinite dimensional version of the well-known Sard--Brown Lemma.

\begin{theorem}[On the density of regular values \cite{Smale}]
If $f\colon \mathcal M \to \mathcal N$ is a Fredholm map between two second countable manifolds, then the set of its regular values is dense in $\mathcal N$.
\end{theorem}

\smallskip
Regarding the notion of Fredholm map one has another useful consequence of the Implicit Function Theorem.

\begin{theorem}[On regularly defined finite dimensional manifolds]
\label{regularly defined finite dimensional manifolds}
Let $f\colon \mathcal M \to \mathcal N$ be a Fredholm map and let $q \in \mathcal N$ be a regular value for~$f$.

Then $f\sp{-1}(q)$, if nonempty, is a manifold whose dimension is the same as the index of $f$.
Moreover, for any $p \in f\sp{-1}(q)$ one has $T_p(f\sp{-1}(q)) = \Ker df_p$.
\end{theorem}

\section{Results}
\label{Results}

Let $G$ and $H$ be two real Hilbert spaces, and consider the problem
\begin{equation}
\label{problem}
\begin{cases}
\;Lx + \e N(x) = \l Cx,\\[.3ex]
\;x \in S,
\end{cases}
\end{equation}
where $\e, \l \in \R$, $L,C\colon G \to H$ are bounded linear operators, and $N\colon S \to H$ is a continuous map defined on the unit sphere of $G$.

A \emph{solution} of \eqref{problem} is a \emph{triple} $(x,\e,\l)$ which satisfies the system, and $(\e,\l)$ is called the \emph{eigenpair} corresponding to the \emph{(unit) eigenvector} $x$.
The solutions and the eigenpairs with $\e=0$ are said to be \emph{trivial}.

By $\Sigma$ we denote the set of the solutions of \eqref{problem} and by $\mathcal E$ we designate its projection into the $\e\l$-plane. Notice that $\mathcal E$ coincides with the set of the eigenpairs of \eqref{problem}.
By $\Sigma_0$ and $\mathcal E_0$ we mean, respectively, the sets of the trivial solutions and the trivial eigenpairs of \eqref{problem}.

\medskip
In order to simplify some statements, it is convenient to introduce the following

\begin{definition}
\label{simple}
A trivial solution $(x_*,0,\l_*)$ of \eqref{problem}, as well as the corresponding eigenpair $(0,\l_*)$, will be called \emph{simple} if the associated linear operator $A = L-\l_* C$ satisfies the following conditions:
\begin{itemize}
\item[(1)]
$\Ker A = \R x_*$,
\item[(2)]
$Cx_*\not= 0$,
\item[(3)]
$\Im A \oplus C (\Ker A) = H$.
\end{itemize}
\end{definition}

To proceed with the analysis of the structure of $\Sigma$, we need a result regarding the linear eigenvalue problem $Lx=\l Cx$. Namely

\begin{lemma}
\label{diffeomorphism}
Assume that $(x_*,0,\l_*)$ is a simple solution of problem \eqref{problem}.
Then, the map
\[
\Psi\colon S\times\R \to H, \quad
(x,\l) \mapsto Lx-\l Cx,
\]
sends, diffeomorphically, a neighborhood of $(x_*,\l_*)$ in $S\times\R$ onto a neighborhood of the origin $0 \in H$.
\end{lemma}

\begin{proof}
Let $z_* = (x_*,\l_*)$ and observe that the tangent space of $S\times\R$ at $z_*$ is the $1$-codimensional subspace
\[
T_{z_*}(S\times\R) = T_{x_*}(S) \times \R = (\R x_*)\sp{\perp} \times \R\]
of $G\times\R$.

Because of the (Local) Inverse Function Theorem, it is enough to prove that the differential $d\Psi_{z_*}\colon T_{z_*}(S\times\R) \to H$ of $\Psi$ at $z_*$ is an isomorphism.

Let us show first that $d\Psi_{z_*}(\dot z) = 0$ implies $\dot z = 0$.
We have
\begin{equation}
\label{in the kernel}
d\Psi_{z_*}(\dot z) =
(L-\l_*C)\dot x - \dot\l Cx_* = 0,
\end{equation}
where $\dot z = (\dot x, \dot \l)$ with $\langle\dot x,x_*\rangle= 0$.

Since $\Ker(L-\l_*C) = \R x_*$, from the splitting condition $(3)$ in Definition \ref{simple} we derive
\begin{equation}
\label{decomposition}
H =
\Im(L-\l_*C) \oplus \R Cx_*.
\end{equation}
Taking into account that $Cx_* \not= 0$, \eqref{in the kernel} implies
$\dot x \in \R x_*$ and $\dot \l = 0$.
Thus, recalling that $\langle \dot x, x_* \rangle = 0$, we get $\dot x = 0$ and, consequently, $\dot z = (\dot x, \dot \l) = 0$.

It remains to show that $d\Psi_{z_*}$ is surjective. This task will be accomplished by showing that the operator $d\Psi_{z_*}$ is Fredholm of index zero.

To this purpose we regard the differential
$
d\Psi_{z_*}\colon (\R x_*)\sp{\perp} \times \R \to H
$
as the sum of two bounded linear operators, $A$ and $B$, defined by
\[
A\dot z = (L-\l_* C)\dot x \quad \text{and} \quad B\dot z = -\dot \l Cx_*,
\]
with $\dot z = (\dot x,\dot\l) \in (\R x_*)\sp{\perp} \times \R$.

Concerning the operator $A$, observe that its kernel is the $1$-dimensional subspace $\{0\}\times \R$ of $(\R x_*)\sp{\perp} \times \R$ and its image coincides with that of $L-\l_* C$, whose codimension is one, due to the splitting \eqref{decomposition} together with the condition $Cx_* \not=0$.
Thus $A$ is Fredholm of index zero and, consequently, so is $d\Psi_{z_*}$, since the operator $B$ is compact.
This completes the proof.
\end{proof}

\smallskip
Lemma \ref{diffeomorphism} allows us to provide an alternative proof of the result of Chiappinelli mentioned in the Introduction. In fact, we get the following corollary, whose proof is only sketched, being essentially the same as that in the finite dimensional setting regarding Corollary 3.3 in \cite{BeCaFuPe-s3}.

\begin{corollary} 
Let $(x_*,0,\l_*)$ be a simple solution of problem \eqref{problem} and assume that $N\colon S \to H$ is a Lipschitz continuous map.
Then there exists a neighborhood $(-\d,\d)$ of $0 \in \R$ and a Lipschitz curve
\[
\e \in (-\d,\d) \mapsto (x(\e),\l(\e)) \in S\times\R
\]
such that $(x(0),\l(0))=(x_*,\l_*)$ and
\[
Lx(\e) + \e N(x(\e)) = \l(\e)C x(\e), \quad \forall\, \e \in (-\d,\d).
\]
\end{corollary} 
\begin{proof}[Sketch of proof]
Lemma \ref{diffeomorphism} shows that the restriction $\Psi|_U$ of $\Psi$ to a convenient neighborhood $U$ of $(x_*,\l_*)$ in $S\times\R$ is a diffeomorphism onto a neighborhood $\Psi(U)$ of $0$ in $H$. Taking $U$ smaller, if necessary, we may assume that this diffeomorphism is Lipschitz, with Lipschitz inverse.
As in \cite{BeCaFuPe-s3}, putting $q=Lx-\l Cx$, one can transform the equation $Lx-\l Cx = -\e N(x)$ into an equivalent fixed point problem of the type $q = \e f(q)$, where $f\colon \Psi(U) \to H$ is Lipschitz continuous, with bounded image.
Thus, if $|\e| < \d$, with $\d$ is sufficiently small, the map $q \mapsto \e f(q)$ is a contraction whose image is contained in a complete subset of $\Psi(U)$, so that one gets a unique fixed point $q(\e)$ of $\e f$. Taking into account that $f$ is dominated by a constant $M$ and is Lipschitz with some constant $L$, standard computations show that, if, in addition, $\d < 1/L$, then the map $\e \mapsto q(\e)$ is Lipschitz with constant $M/(1-\d L)$.
This implies that the function $\e \in (-\d,\d) \mapsto (\Psi|_U)\sp{-1}(q(\e))$ is Lipschitz, since so is $(\Psi|_U)\sp{-1}$.
Finally, the curve $\e \mapsto (x(\e),\l(\e)) := (\Psi|_U)\sp{-1}(q(\e))$ verifies the assertion.
\end{proof}

Before proving our main theorem about global persistence we need some preliminary results regarding some crucial properties of the map
\[
\Phi\colon S\times\R\times\R \to H, \quad (x,\e,\l) \mapsto Lx+\e N(x)-\l Cx,
\]
whose set of zeros, $\Phi\sp{-1}(0)$, coincides with $\Sigma$.

Incidentally, we observe that the function $\Psi\colon S\times\R \to H$ defined in Lemma \ref{diffeomorphism} is the partial map of $\Phi$ corresponding to $\e=0$.
In other words, $\Psi$ can be, and occasionally will be, identified with the restriction of $\Phi$ to the subset $Z = S\times \{0\} \times \R$ of the domain $S\times\R\times\R$ of $\Phi$.
Because of this identification, the set $\Sigma_0=Z \cap \Sigma$ of the trivial solutions of \eqref{problem} may be regarded as $\Psi\sp{-1}(0)$.

We point out that $S\times\R\times\R$ is a $1$-codimensional submanifold of the Hilbert space $G\times\R\times\R$ and $Z$ is a $1$-codimensional submanifold of $S\times\R\times\R$.

\medskip
The next lemma provides conditions on the operators $L$, $C$ and $N$ ensuring the properness of the map $\Phi$ on bounded and closed subsets of $S\times\R\times\R$.

\begin{lemma}
\label{properness}
Regarding problem \eqref{problem}, assume that $C$ and $N$ are compact, and that $L$ is a Fredholm operator.

Then the map $\Phi$ is proper on bounded closed sets.
\end{lemma}
\begin{proof}
Let $K$ be a compact subset of $H$ and $D$ a bounded and closed subset of $S\times\R\times\R$.
We need to show that $D \cap \Phi\sp{-1}(K)$ is a compact set.
To this purpose, observe that the function $\Phi$ is the sum of three maps: $\mathcal L$, $\mathcal N$ and $\mathcal C$, given by $\mathcal L(p)=Lx$, $\mathcal N(p)=\e N(x)$ and $\mathcal C(p)=-\l Cx$, where $p=(x,\e,\l)$.

The operator $L$, being Fredholm, is proper on bounded closed sets. Consequently, as one can easily check, the map $\mathcal L$ has the same property.
Since, by assumption, the operators $C$ and $N$ are compact, so are the corresponding maps $\mathcal C$ and $\mathcal N$. Thus, $\mathcal C(D)$ and $\mathcal N(D)$ are relatively compact in $H$. Hence, the set $K - \mathcal N(D) - \mathcal C(D)$ is contained in a compact subset $\mathcal K$ of $H$.
The assertion now follows from the fact that $D \cap \Phi\sp{-1}(K)$ is a closed subset of the compact set $D \cap {\mathcal L}\sp{-1}(\mathcal K)$.
\end{proof}

The next result provides sufficient conditions for the nonlinear map $\Phi$ to be Fredholm. In order to make the statement meaningful, $N$ is assumed to be smooth.

\begin{lemma}
\label{Fredholm map}
Under the assumptions of Lemma \ref{properness} suppose, in addition, that $N$ is smooth and that the index of $L$ is zero.

Then the map $\Phi$ is Fredholm of index $1$.
\end{lemma}
\begin{proof}
Let $\mathcal L,\, \mathcal N,\, \mathcal C\,\colon S\times\R\times\R \to H$ be as in the proof of Lemma \ref{properness}.
Since the smooth maps $\mathcal N$ and $\mathcal C$ are compact, so are their differentials at any $p \in S\times\R\times\R$.
Therefore, recalling that $\Phi = \mathcal L + \mathcal N + \mathcal C$, it is enough to show that $\mathcal L$ is Fredholm of index~$1$.
To this purpose observe that $\mathcal L$ is the composition of three maps: the projection $P\colon S\times\R\times\R \to S$, the inclusion $J\colon S \hookrightarrow G$ and the linear operator $L\colon G \to H$.
The first two maps, $P$ and $J$, are Fredholm of index $2$ and $-1$, respectively.
As $L$ has index $0$, the composite map $\mathcal L = L\circ J \circ P$ is Fredholm of index $2-1+0 = 1$.
\end{proof}

The next result is crucial in the proof of Theorem \ref{continuation}.

\begin{lemma}
\label{N smooth}
Let $p_* = (x_*,0,\l_*)$ be a simple solution of problem \eqref{problem}.
Assume that the operator $C$ is compact and that $N$ is compact and smooth.

Then, given a neighborhood $U$ of $p_*$ in $Z = S\times \{0\} \times \R$, there exists a neighborhood $V$ of $0 \in H$ such that, if $q \in V$ is a regular value for $\Phi$, the set $\Phi\sp{-1}(q)$ contains a smooth boundaryless connected curve that intersects $U$ and is either unbounded or diffeomorphic to a circle containing at least two points of $Z$.
\end{lemma}

\begin{proof}
Notice that, if the assertion holds for the neighborhood $U$, then it holds as well for any $W$ such that $U \subseteq W \subseteq Z$.
Therefore, recalling that the restriction of $\Phi$ to $Z$ can be identified with the map $\Psi$ of Lemma \ref{diffeomorphism}, we may assume that $U$ is so small that it is mapped by $\Phi$ diffeomorphically onto a neighborhood $V$ of $0 \in H$.
Moreover, as $p_*$ is simple, the operator $A=L-\l_*C$ is Fredholm of index zero. Consequently, so is $L$, because of the compactness of $C$.

Let $q \in V$ be a regular value for $\Phi$.
Since, according to Lemma \ref{Fredholm map}, the map $\Phi$ is Fredholm of index $1$, the set $\Phi\sp{-1}(q)$ is a smooth (boundaryless) $1$-dimensional submanifold of $S\times\R\times\R$ (see Theorem \ref{regularly defined finite dimensional manifolds}).
Any component of $\Phi\sp{-1}(q)$ is a closed subset of $S\times\R\times\R$, which is either compact, and therefore diffeomorphic to a circle, or non-compact, and consequently diffeomorphic to the open interval $(0,1)$, according to Theorem \ref{one-dimensional manifolds}.

Denote by $\Gamma$ the component of $\Phi\sp{-1}(q)$ containing the unique point $p \in U$ such that $\Phi(p) = q$.
In particular the assertion that $\Gamma \cap U$ is nonempty is verified.

Now, if $\Gamma$ is unbounded, the proof is completed.
Suppose, on the contrary, that it is bounded.
Then, due to the properness of $\Phi$ on bounded and closed sets ensured by Lemma \ref{properness}, the component $\Gamma$ is compact and, therefore, diffeomorphic to the unit circle $S\sp{1}$.

Therefore, to complete the proof, it is sufficient to show that $\Gamma$ intersects $Z$ at some point different from $p$.
Namely, it is enough to prove that the continuous function
$
\s\colon (x,\e,\l) \in \Gamma \mapsto \e \in \R
$
vanishes at some point of $\Gamma\setminus \{p\}$.

To this purpose we will show that the disjoint open subsets
\[
\Gamma_- = \s\sp{-1}((-\infty,0)) \quad \text{and} \quad 
\Gamma_+ = \s\sp{-1}((0,+\infty))
\]
of the connected set $\Gamma\setminus \{p\}$ are both nonempty; and this will be obvious if we show that the intersection at $p$ between the curve $\Gamma$ and the manifold $Z$ is transversal, which implies that $\s$ has a sign-jump at $p$. 
Indeed, the transversality is a consequence of the fact that the point $p$, apart of being regular for the map $\Phi$, is as well regular for the restriction of $\Phi$ to $U$, due to the diffeomorphism $\Phi|_U\colon U \to V$. In fact, as $p$ is regular for $\Phi$, one has $T_p(\Gamma) = \Ker d\Phi_p$ (Theorem \ref{regularly defined finite dimensional manifolds}); and this $1$-dimensional space is not contained in $T_p(Z)$, since the operator $d(\Phi|_U)_p\colon T_p(Z) \to H$, which coincides with the restriction of $d\Phi_p$ to $T_p(Z)$, is injective (it is actually invertible).

In conclusion, the union of the open sets $\Gamma_-$ and $\Gamma_+$ cannot coincide with che connected set $\Gamma\setminus \{p\}$. Therefore, the function $\s$ vanishes at some point of $\Gamma\setminus \{p\}$; and this concludes the proof.
\end{proof}

The next lemma is a Whyburn-type topological result which is crucial in the proof of Theorem \ref{continuation}.
\begin{lemma}[\cite{FuPe}]
\label{Whyburn}
Let $Y_0$ be a compact subset of a locally compact metric space~$Y$.
Assume that every compact subset of $Y$ containing $Y_0$ has nonempty boundary.
Then $Y \backslash Y_0$ contains a connected set whose closure in $Y$ is non-compact and intersects~$Y_0$.
\end{lemma}

Recalling that $\Sigma$ denotes the set of solutions of problem \eqref{problem} and that $\Sigma_0$ is its subset of the trivial ones, the following simple result provides an example of locally compact space, needed in the proof of Theorem \ref{continuation}.

\begin{lemma}
\label{Y locally compact}
Assume that the operators $N$ and $C$ are compact and that $p_*$ is a simple solution of problem \eqref{problem}.

Then, the metric subspace $Y = (\Sigma \setminus \Sigma_0) \cup \{p_*\}$ of $\Sigma$ is locally compact.
\end{lemma}

\begin{proof}
Observe that $\Sigma_0 = Z \cap \Sigma$, where $Z = S\times \{0\}\times \R$.
Therefore $\Sigma_0$ is closed in the metric space $\Sigma = \Phi\sp{-1}(0)$, which is locally compact because of Lemma \ref{properness}.
Clearly $Y$ coincides with $\Sigma \setminus (\Sigma_0 \setminus \{p_*\})$. Moreover $\Sigma_0 \setminus \{p_*\}$ is closed in $\Sigma$, the point $p_*$ being isolated in the closed set $\Sigma_0$ because of Lemma \ref{diffeomorphism}.
Thus, the metric space $Y$, as an open subset of $\Sigma$, is locally compact.
\end{proof}

From Lemmas \ref{Whyburn} and \ref{Y locally compact} we derive the following

\begin{lemma}
\label{metric pair}
Let the operators $N$ and $C$ be compact, and let $p_*$ be a simple solution of problem \eqref{problem}.
Assume that any compact subset of $Y = (\Sigma \setminus \Sigma_0) \cup \{p_*\}$ containing $p_*$ has nonempty boundary in $Y$.

Then $\Sigma \setminus \Sigma_0$ has a connected set whose closure in $\Sigma$ contains $p_*$ and is either unbounded or meets a trivial solution $p\sp* \not= p_*$.
\end{lemma}

\begin{proof}
Because of Lemma \ref{Y locally compact} and our assumption, Lemma \ref{Whyburn} applies to the metric pair $(Y,Y_0)$, $Y_0=\{p_*\}$.
Consequently, $Y \backslash Y_0 = \Sigma\backslash \Sigma_0$ has a connected subset, say $\mathcal D$, whose closure $\bar{\mathcal D} \cap Y$ in $Y$ is non-compact and contains $p_*$ (here $\bar{\mathcal D}$ denotes the closure of $\mathcal D$ in $\Sigma$ or, equivalently, in $S\times\R\times\R$).

Since $p_*$ belongs to the connected set $\bar{\mathcal D}$ (actually $p_* \in \bar{\mathcal D} \cap Y$), it is sufficient to show that, if $\bar{\mathcal D}$ is bounded (hence compact, because of Lemma \ref{properness}), then it must intersect $\Sigma_0\setminus \{p_*\}$, and this is clearly true since otherwise $\bar{\mathcal D}$ would coincide with $\bar{\mathcal D} \cap Y$, which is non-compact.
\end{proof}

We are now in a position to prove our main result regarding problem \eqref{problem}.

\begin{theorem} 
[Global continuation of solution triples]
\label{continuation}
Regarding problem \eqref{problem}, assume that the spaces $G$ and $H$ are separable, that the operators $N$ and $C$ are compact, and that $p_*=(x_*,0,\l_*)$ is a simple solution.

Then, the set $\Sigma \setminus \Sigma_0$ of the nontrivial solutions has a connected subset whose closure in $\Sigma$ contains $p_*$ and is either unbounded or meets a trivial solution $p\sp* \not= p_*$.
\end{theorem} 

\begin{proof}
Denoting by $Y$ the locally compact metric space $(\Sigma \setminus \Sigma_0) \cup \{p_*\}$, it is sufficient to apply Lemma \ref{metric pair} by proving that \textsl{any compact subset of $Y$ containing $p_*$ has nonempty boundary in $Y$}.

By contradiction, assume there exists a compact subset $K$ of $Y$ containing $p_*$ whose boundary, in $Y$, is empty.
This compact set is open in $Y$, therefore it is far away from its (relatively) closed complement $Y\setminus K$.
Actually, it is also far away from $\Sigma\setminus K = (Y\setminus K) \cup( \Sigma_0\setminus \{p_*\})$, this set being disjoint from $\Sigma_0\setminus \{p_*\}$, which is closed according to Lemma \ref{diffeomorphism}.
Consequently, there exists a bounded open subset $W$ of $S\times\R\times\R$ such that $W\cap \Sigma= K$ and
\begin{equation}
\label{empty}
\partial W\cap \Sigma=\emptyset.
\end{equation}
Because of Lemma \ref{diffeomorphism}, we may suppose that the intersection $U = Z \cap W$ is mapped by $\Phi$ diffeomorphically onto a neighborhood $V$ of the origin $0 \in H$.

Since $G$ is separable, given $\d > 0$, there exists a smooth map $N_\d \colon S \to H$ with finite dimensional image and such that $\|N(x)-N_\d(x)\| < \d$ for all $x \in S$.
In fact, a well known result in separable Hilbert spaces (see, for example, \cite{Eells} and references therein) ensures the existence of a smooth $(\d/2)$-approximation $\hat N_\d$ of $N$, so that the required map $N_\d$ is obtained by composing $\hat N_\d$ with the orthogonal projection of $H$ onto the vector space spanned by a $(\d/2)$-net of the totally bounded set $N(S)$.
Now, define $\Phi_\d \colon S\times\R\times\R \to H$ by $\Phi_\d(x,\e,\l) = (L-\l C)x + \e N_\d(x)$.
This map is a nonlinear Fredholm operator between two separable Hilbert manifolds. Therefore, the celebrated Sard--Smale result \cite{Smale} implies the existence of a regular value $q_\d \in V$ of $\Phi_\d$ such that $\|q_\d\| < \d$.
By Lemma \ref{N smooth} we deduce that $\Phi\sp{-1}_\d(q_\d)$ has a connected subset $\Gamma_\d$ that intersects $U$ and is either unbounded or contains at least two points of $Z$.
Since $\Phi$ maps $U$ diffeomorphically onto $V$ (and coincides with $\Phi_\d$ on $U$), one and only one of the points of $\Gamma_\d$ lies in $U$.
Therefore, in any of the two cases the connected set $\Gamma_\d$ must have points outside the bounded set $W$ and, consequently, must contain at least one point $p_\d \in \partial W$.

Now, denoting $p_\d = (x_\d, \e_\d, \l_\d)$, we have
\[
\|\Phi(p_\d)\| \leq \|\Phi_\d(p_\d)\| + \|\Phi(p_\d)- \Phi_\d(p_\d)\| = \|q_\d\| + |\e_\d|\|N(x_\d)-N_\d(x_\d)\| \,\leq \d +c\d,
\]
where $c = \sup\{|\e| : (x,\e,\l) \in \partial W\}$.

As a consequence we get $\inf\{\|\Phi(p)\|: p \in \partial W\} = 0$.
Thus, the properness of $\Phi$ on $\partial W$ (ensured by Lemma \ref{properness}) implies the existence of a compact subset of $\partial W$ in which the infimum (and therefore the minimum) of the real functional $\|\Phi\|$ is zero (to see this consider a minimizing sequence of $\|\Phi\|$ in $\partial W$, then add $0 \in H$ to the image of this sequence in order to get a compact subset $\mathcal K$ of $H$, and take the set $\Phi\sp{-1}(\mathcal K) \cap \partial W$).

Hence, we obtain $\partial W\cap \Phi\sp{-1}(0)=\partial W\cap \Sigma\not=\emptyset$, contradicting \eqref{empty}, so that Lemma \ref{metric pair} applies.
\end{proof}

\begin{remark}
\label{the component}
Let $p_* = (x_*,0,\l_*)$ be a trivial solution of problem \eqref{problem}, and assume that the connected component of $\Sigma$ containing $p_*$, call it $\Gamma$, is bounded.

Under the assumptions of Theorem \ref{continuation}, taking into account that the closure of a connected set is still connected, one gets that
\begin{equation}
\label{gammaassert}
\text{$\Gamma$ meets a point $p\sp* = (x\sp*,0,\l\sp*)$ different from $p_*$}.
\end{equation}

Consequently, one necessarily has $x^* \not= x_*$, since otherwise the condition $Cx_* \not= 0$ would imply $\l^*=\l_*$, contradicting $p^* \not= p_*$.
Obviously, one could have $\l^*= \l_*$, but in this case, $p_*$ being simple, one would obtain $x^*= -x_*$. 

We point out that assertion \eqref{gammaassert} is trivially satisfied even in most cases in which $p_*$ is not simple. Namely, whenever $\dim (\Ker(L-\l_* C)) > 1$, since $\Gamma$ contains (and possibly coincides with) the geometric sphere
\[
\big(\Ker(L-\l_* C) \cap S\big)\times \{(0,\l_*)\},
\]
made up of infinitely many trivial solutions, all having the same eigenvalue $\l_*$.

Neverthless, when $p_*$ is simple, Lemma \ref{diffeomorphism} ensures that it is isolated in the set $\Sigma_0$ of the trivial solutions; consequently, assertion \eqref{gammaassert} implies that $\Gamma$ is not contained in $\Sigma_0$.
\end{remark}

Under the assumptions of our main result, Theorem \ref{continuation of eigenpairs intro} ensures that,
\emph{in the set $\mathcal E$ of all the eigenpairs of \eqref{problem}, the connected component containing the simple eigenpair $(0,\l_*)$ is either unbounded or meets a trivial eigenpair $(0,\l\sp*)$ with $\l\sp* \not= \l_*$.}
For this reason we are inclined to believe that Theorem \ref{continuation} could be extended according to Conjecture \ref{conjecture-intro}.

Sometimes, a bounded linear operator $A\colon G \to H$ acting between real Hilbert spaces is easily recognized as Fredholm of index zero. For example, this happens when $G$ and $H$ have the same finite dimension, or when $A$ is a compact linear perturbation of one of the following operators (see e.g.~\cite{Taylor & Lay}):
\begin{itemize}
\item
an invertible operator $L\colon G \to H$;
\item
the restriction $L$ of a surjective operator $\hat L\colon \hat G \to H$ to a closed subspace $G$ of $\hat G$, whose codimension is finite and the same as $\Ker\hat L$ ($x\in G$ can be regarded as a sort of boundary condition).
\item
a self-adjoint operator $L\colon G \to G$ for which $0 \in \R$ is an isolated eigenvalue of finite multiplicity.
\end{itemize}

In these cases, in order to verify that a trivial solution of problem \eqref{problem} is simple, it may be convenient to take into consideration the following

\begin{proposition} 
\label{unsolvability}
Let $(x_*,0,\l_*)$ be a trivial solution of problem \eqref{problem}.
Assume that $\l_* \in \R$ and $x_* \in S$ are such that
\begin{itemize}
\item[(0)]
$A = L-\l_*C$ is Fredholm of index zero
\item[(1)]
$\Ker A = \R x_*$.
\end{itemize}
Then, $(x_*,0,\l_*)$ is simple if and only if the equation
$
Ax = Cx_*
$
is unsolvable in $G$.
\end{proposition} 

\begin{proof}
Assume first that the equation $Ax = Cx_*$ has no solutions.
Thus, condition $(2)$ of Definition \ref{simple} is satisfied (otherwise the above equation would admit the trivial solution) and, obviously, $Cx_* \not\in \Im\!{A}$.
Consequently, since conditions $(0)$ and $(1)$ imply that $\Im\!A$ has codimension one, the one-dimensional subspace $\R Cx_* = C(\Ker A)$ of $H$ intersects $\Im A$ transversally, and this proves that condition $(3)$ of Definition \ref{simple} is as well satisfied.

Conversely, since, as already pointed out, $\Im A$ has codimension one, conditions (1) and (3) of Definition \ref{simple} imply that $Cx_* \not\in \Im A$; that is, $Ax = Cx_*$ has no solutions.
\end{proof}

\smallskip
We close this section with a consequence of Theorem \ref{continuation} regarding the perturbed finite dimensional classical eigenvalue problem.

\begin{corollary} 
[Global continuation in finite dimension \cite{BeCaFuPe-s3}]
\label{continuation: finite dimension}
In problem \eqref{problem}, let $G=H$ be finite dimensional.
Suppose that $C$ is the identity $I$ of $G$, and let $x_* \in S$ be an eigenvector of $L$ corresponding to a simple eigenvalue $\l_*$.

Then, the trivial solution $p_*=(x_*,0,\l_*)$ of \eqref{problem} is simple and, consequently, the set $\Sigma \setminus \Sigma_0$ has a connected subset whose closure contains $p_*$ and is either unbounded or meets a trivial solution $p\sp* \not= p_*$.

\end{corollary} 

\begin{proof}
 The operator $A = L-\l_*I$ is Fredholm of index zero and the condition $Cx_* = x_* \not= 0$ is trivially satisfied.
Therefore, according to Proposition \ref{unsolvability}, $(x_*,0,\l_*)$ is simple if (and only if) $x_* \not\in \Im A$; and this is true, otherwise $\Ker A\sp2$ would differ from $\Ker A$ and the eigenvalue $\l_*$ would not be simple.
The final assertion now follows from Theorem \ref{continuation}.
\end{proof}

\section{Examples}
\label{Examples}

In this last section we will see, in some examples, how Theorem \ref{continuation} applies.
In particular we will show that, in our main result, the hypothesis that the ``starting'' trivial solution is simple cannot be removed.
In each example the real Hilbert spaces $G$ and $H$, as well as the operators $L$, $N$ and $C$, will be explicitly introduced.
The norm in any Hilbert space will be the standard one associated with the inner product.
By $S$, $\Sigma$ and $\mathcal E$ we shall always mean, respectively, the unit sphere of $G$, the set of solutions of the given problem, and the set of the corresponding eigenpairs (which, we recall, is the projection of $\Sigma$ into $\R\sp2$).
As in the previous section, $\Sigma_0$ is the subset of $\Sigma$ of the trivial solutions.

\medskip
We begin with an elementary example regarding a perturbed two-dimensional classical eigenvalue problem, whose set $\Sigma$ is a (topological) circle and connects all the trivial solutions which, in this case, are four (two for each eigenvalue of the unperturbed problem) and all simple.
As one can check, the projection of $\Sigma$ onto $\mathcal E$ is a double covering map.

\begin{example} 
\label{example1}
Let $G=H=\R\sp2$ and consider the problem
\begin{equation}
\label{problem in E1}
\left\{
\begin{array}{ccc}
x_1 + \e x_2 \eql \l x_1,\\
-x_2 - \e x_1 \eql \l x_2,\\[0pt]
x_1\sp2+x_2\sp2 \eql 1,
\end{array}\right.
\end{equation}
in which the operators $L$ and $N$ are defined by sending $x = (x_1,x_2) \in G$ into $(x_1,-x_2)$ and $(x_2,-x_1)$, respectively, and $C$ is the identity.

The operator $L$ has two simple eigenvalues: $\l_* = 1$ and $\l\sp* = -1$.
Consequently, one gets four trivial solutions of the above problem (all of them simple):
\[
\big(\pm(1,0),0,1\big) \quad \text{and} \quad \big(\pm(0,1),0,-1\big).
\]
Notice that the set $\mathcal E$ of the eigenpairs of \eqref{problem in E1} is the unit circle $\e\sp2+\l\sp2 = 1$.
Therefore, $\Sigma$ is necessarily bounded.

The eigenpairs $(\e,\l)$ can be represented, parametrically, as $(-\sin \t, \cos \t)$, with $\t \in [0,2\pi]$ and, as one can check, given $\t$, the kernel of the linear operator
\[
L - (\sin{\t})N -(\cos{\t})C
\]
is spanned by the unit vector
\[
x(\t) =(x_1(\t),x_2(\t)) = (\cos(\t/2),\sin(\t/2)).
\]
Thus, $\Sigma$ can be parametrized as follows:
\[
\t \in [0,4\pi] \mapsto (x(\t),-\sin\t,\cos\t),
\]
and this shows that $\Sigma$ is a topological circle.

As one can easily check, $\Sigma$ encounters all the four trivial solutions of problem \eqref{problem in E1}. Moreover, the set $\Sigma \setminus \Sigma_0$ of the nontrivial solutions has four connected components, each of them diffeomorphic to an open real interval and satisfying the assertion of Conjecture \ref{conjecture-intro} (and, consequently, the thesis of Theorem \ref{continuation}).

Incidentally, we observe that the projection of $\Sigma$ onto the circle $\mathcal E$ is a double covering map.
\end{example}

\medskip
The following illustrating example regards a problem in which the unperturbed equation has a unique eigenvalue.
Since, as we shall see, the assumptions of Theorem \ref{continuation} are satisfied, according to Theorem \ref{continuation of eigenpairs intro}, the set $\mathcal E$ of the eigenpairs is unbounded and, consequently, so is the set $\Sigma$ of the solution triples.

\begin{example} 
\label{example2}
Let $H\sp1([0,2\pi],\R)$ denote the space of all the absolutely continuous functions $x\colon [0,2\pi] \to \R$ whose derivative is in $L\sp2([0,2\pi],\R)$. This is a separable real Hilbert space with inner product
\[
\langle x,y \rangle_1 = \frac{1}{2\pi}\int_0\sp{2\pi}\pb\big(x(s)y(s) + x'(s)y'(s)\big)\, ds.
\]
Our source space $G$ is the kernel of the continuous functional
\[
x \in H\sp1([0,2\pi],\R) \mapsto x(2\pi)-x(0) = \int_0\sp{2\pi}\pb x'(s)\,ds.
\]
So that $G$ is a closed, $1$-codimensional subspace of $H\sp1([0,2\pi],\R)$.
The target space $H$ is $L\sp2([0,2\pi],\R)$ with inner product
\[
\langle x,y \rangle = \frac{1}{2\pi}\int_0\sp{2\pi} \pb x(s)y(s)\, ds.
\]
Consider the problem
\begin{equation}
\label{problem in E2}
\begin{cases}
x'(t) +\e \sin t = \l x(t),\\
x(0)=x(2\pi),\\
\langle x,x \rangle_1 = 1.
\end{cases}
\end{equation}
In abstract form, this can be written as
\begin{equation*}
\label{abstract problem in E2}
\begin{cases}
Lx + \e N(x) = \l Cx,\\
x \in S,
\end{cases}
\end{equation*}
where
\begin{itemize}
\item
$L\colon G \to H$ is the derivative $x \mapsto x'$;
\item
$N\colon G \to H$ is the (constant) map defined by $N(x)=\sin(\cdot)$;
\item
$C$ is the (compact) inclusion of $G$ into $H$.
\end{itemize}

The unperturbed equation $Lx =\l Cx$ has a unique eigenvalue, $\l_* = 0$, whose corresponding eigenspace, the kernel of $L$, is the set of constant functions. Therefore, our problem has only two trivial solutions: $(\pm x_*,0,0)$, where $x_* \in S$ is the constant function $t \mapsto 1$.

The operator $A = L-\l_*C = L$ is Fredholm of index zero, since it is the composition of the inclusion $G \hookrightarrow H\sp1([0,2\pi],\R)$, which is Fredholm of index $-1$, with the differential operator $x \in H\sp1([0,2\pi],\R) \mapsto x' \in H$, which is Fredholm of index $1$ (being surjective with $1$-dimensional kernel).
Therefore, $A$ satisfies the conditions (0) and (1) of Proposition \ref{unsolvability}.
Hence, the trivial solution $(x_*,0,0)$ is simple provided that the problem
\begin{equation*}
\begin{cases}
x'(t) = 1,\\
x(0)=x(2\pi)
\end{cases}
\end{equation*}
has no solutions, and this is clearly true.
Thus, Theorem \ref{continuation} applies.
Consequently, according to Remark \ref{the component}, the connected component $\Gamma$ of $\Sigma$ containing $(x_*,0,0)$ is either unbounded or meets $(-x_*,0,0)$.
As we shall see, $\Gamma$ has both the properties.

Standard computations show that, given any $(\e,\l)$, the problem
\begin{equation}
\label{problem 2 in E2}
\begin{cases}
x'(t) +\e \sin t = \l x(t),\\
x(0)=x(2\pi)
\end{cases}
\end{equation}
has a solution
\[
x(t) = \frac{\e}{1+\l\sp2}\big(\l \sin t + \cos t\big),
\]
which is unique if and only if $\l \not= 0$, and its norm is
$
|\e|/\sqrt{1+\l\sp2}.
$

Regarding the case $\l=0$, given any $\e$, problem \eqref{problem 2 in E2} has infinitely many solutions: $x(t) = c + \e \cos t$ ($c \in \R$), with norm $\sqrt{c\sp2+\e\sp2}$.

Therefore, the set $\mathcal E$ of the eigenpairs is the union of three connected sets.
One is the \emph{segment} $[-1, 1] \times \{0\}$, corresponding to the case $\l=0$.
The other two are the \emph{left} and \emph{right branches} of the hyperbola
\[
\e\sp2 -\l\sp2 = 1.
\]
Thus, $\mathcal E$ is unbounded and, consequently, so is the set $\Sigma$ of all the solution triples, $\mathcal E$ being its projection into $\R\sp2$.

Apart of being unbounded, the set $\mathcal E$ is connected, since both the branches of the hyperbola have a point in common with the segment $[-1, 1] \times \{0\}$. These points are $(-1,0)$ for the left branch and $(1,0)$ for the right one.

Let us show that $\Sigma$ is as well connected and, consequently, coincides with the component $\Gamma$ containing the trivial solution $(x_*,0,0)$.

Obviously, $\Sigma$ is the union of three sets, each of them ``over'' one of the following three subsets of $\mathcal E$: the segment $[-1, 1] \times \{0\}$, and the left and right branches of the hyperbola $\e\sp2 -\l\sp2 = 1$.

The bounded set over $[-1, 1] \times \{0\}$, say $\Gamma_s$, regards the solutions $x(t) = c + \e \cos t$, corresponding to the case $\l=0$ and having norm $\sqrt{c\sp2+\e\sp2} = 1$.
Putting $\e = \sin\t$ and $c = \cos\t$, with $\t \in [0,2\pi]$, we get the following parametrization of $\Gamma_s$:
\[
\t \in [0,2\pi] \mapsto (\cos\t+\sin\t\cos(\cdot), \sin\t, 0),
\]
which shows that this set is a topological circle, which meets the two trivial solutions $(\pm x_*,0,0)$, both simple and corresponding to the eigenvalue $\l_*=0$ of $L$.

Notice that, if the bounded and connected set $\Gamma\!_s$ were a component of $\Sigma$, then $\Gamma$ would coincide with $\Gamma\!_s$, and this fact, although compatible with Theorem \ref{continuation} (see Remark \ref{the component}), would contradict Conjecture \ref{conjecture-intro}.
We will show that this is not the case: as we shall see, $\Gamma$ is unbounded, in accord with our conjecture.

The unbounded sets over the branches of the hyperbola, call them $\Gamma\!_l$ and $\Gamma\!_r$, can be parametrized as follows:
\[
s \in \R \mapsto \Big(\mp\frac{1}{\sqrt{1+s\sp2}}\big(s\sin(\cdot) + \cos(\cdot)\big),\mp \sqrt{1+s\sp2},s \Big);
\]
showing that they are both connected (more precisely, diffeomorphic to $\R$).

As one can check, $\Gamma_s$ has the point $(-1,-1,0)$ in common with $\Gamma\!_l$ and the point $(1,1,0)$ in common with $\Gamma_r$. Thus, the set $\Sigma = \Gamma\!_s \cup \Gamma\!_l \cup \Gamma\!_r$ turns out to be connected.
Therefore, the component $\Gamma$ containing the simple solution $(x_*,0,0)$ coincides with the unbounded set $\Sigma$, in accord with Remark \ref{the component}.

Finally, we observe that, in accord with Theorem \ref{continuation}, the set $\Sigma \setminus \Sigma_0$ is the union of two connected components, both unbounded, whose closure of each of them contains the two simple solutions.
\end{example} 

\medskip
The following is an example of a linear system of two coupled differential equations with periodic boundary conditions whose set $\Sigma$ has a bounded component $\Gamma$, diffeomorphic to a circle and containing four trivial solutions, all of them simple.

In addition to $\Gamma$, the other components of $\Sigma$ are infinitely many geometric circles contained in $\Sigma_0$, each of them corresponding to an isolated trivial eigenpair.
This shows that, in Theorem \ref{continuation}, the hypothesis that the ``starting'' solution $(x_*,0,\l_*)$ is simple cannot be removed.

\begin{example} 
\label{coupled equations}
Consider the following system of coupled differential equations with $2\pi$-periodic boundary conditions:
\begin{equation}
\label{system 3 in E3}
\begin{cases}
x'_1(t) + x_1(t) - \e x_1(t) = \l x_2(t),\\
x'_2(t) - x_2(t) - \e x_2(t) = -\l x_1(t),\\
x_1(0)=x_1(2\pi), \; x_2(0)= x_2(2\pi).
\end{cases}
\end{equation}
Let $H\sp1([0,2\pi],\R\sp2)$ denote the real Hilbert space of the absolutely continuous functions
$
x=(x_1,x_2)\colon [0,2\pi] \to \R\sp2
$
whose derivative is in $L\sp2([0,2\pi],\R\sp2)$.

In this example, the source space $G$ is the closed subspace of $H\sp1([0,2\pi],\R\sp2)$ of the functions satisfying the periodic condition $x(0)=x(2\pi)$.
The target space $H$ is $L\sp2([0,2\pi],\R\sp2)$.

Denoting by $a \cdot b$ the standard dot product of two vectors $a=(a_1,a_2)$ and $b=(b_1,b_2)$ of $\R\sp2$, the inner product of $x$ and $y$ in $H$ is given by
\[
\langle x,y \rangle = \frac{1}{2\pi}\int_0\sp{2\pi}x(t)\cdot y(t)\, dt,
\]
while the inner product of $x$ and $y$ in $G$ is
\[
\langle x,y \rangle_1 = \frac{1}{2\pi}\int_0\sp{2\pi}\big(x(t)\cdot y(t) + x'(t)\cdot y'(t)\big)\, dt.
\]
Since $G$ has codimension $2$ in $H\sp1([0,2\pi],\R\sp2)$, the operator $L \colon G \to H$, given by $(x_1,x_2) \mapsto (x_1'+x_1,x_2'-x_2)$, is Fredholm of index zero. In fact, it is the restriction to $G$ of a surjective operator defined on $H\sp1([0,2\pi],\R\sp2)$ whose kernel is $2$-dimensional.

The operators $N$ and $C$ are defined as $(x_1,x_2) \mapsto (-x_1,-x_2)$ and $(x_1,x_2) \mapsto (x_2,-x_1)$, respectively. They are compact, due to the compact inclusion
\[
H\sp1([0,2\pi],\R\sp2) \hookrightarrow L\sp2([0,2\pi],\R\sp2).
\]
As in the previous examples we seek for solutions of \eqref{system 3 in E3} in the sphere $S$ of $G$. That is, we consider the problem
\begin{equation}
\label{abstract problem in E3}
\begin{cases}
Lx + \e N(x) = \l Cx,\\
x \in S.
\end{cases}
\end{equation}
The system of two coupled differential equations
\begin{equation}
\label{system in E3}
\begin{cases}
x'_1(t) + x_1(t) - \e x_1(t) = \l x_2(t),\\
x'_2(t) - x_2(t) - \e x_2(t) = -\l x_1(t)
\end{cases}
\end{equation}
\smallskip
can be represented in a matrix form as
\smallskip
\[
\left(
\begin{array}{c}
x_1'(t) \\[2ex]
x_2'(t)
\end{array}
\right)
=
\left(
\begin{array}{cc}
\e-1 & \l\\[2ex]
-\l & \e+1
\end{array}
\right)
\left(
\begin{array}{c}
x_1(t) \\[2ex]
x_2(t)
\end{array}
\right),
\]
where, given $\e$ and $\l$, the eigenvalues of the matrix
\[
M(\e,\l)
=
\left(
\begin{array}{cc}
\e-1 & \l\\[2ex]
-\l & \e+1
\end{array}
\right)
\]
are $\e \pm \sqrt{1-\l\sp2}$.
Therefore, if $|\l| > 1$, \eqref{system in E3} admits non-zero $2\pi$-periodic solutions if and only if $\e=0$ and $\sqrt{\l\sp2-1} \in \N$; these solutions are oscillating and, with the addition of the trivial one, they form a two-dimensional subspace of $H\sp1([0,2\pi],\R\sp2)$.
While, if $|\l| \leq 1$, \eqref{system in E3} has non-zero $2\pi$-periodic solutions if and only if $\e\sp2+\l\sp2=1$; these solutions are constant and, with the trivial one, constitute a one-dimensional space.

Clearly, the set $\mathcal E$ of the eigenpairs of \eqref{abstract problem in E3} is
\[
\big\{(\e,\l) \in \R\sp2: \e\sp2+\l\sp2=1\big\} \cup
\big\{(0,\l) \in \R\sp2: \l = \pm \sqrt{1+n\sp2},\; n \in \N \big\}.
\]

As in the Example \ref{example1}, the eigenpairs of the circle $\{(\e,\l) \in \R\sp2: \e\sp2+\l\sp2=1\}$ can be represented parametrically. In this case we set $(\e,\l)=(\cos \a, \sin \a)$, with $\a \in [0,2\pi]$. Thus, as one can show, given any $\a$, the kernel of the linear operator
\[
L + (\cos{\a})N -(\sin {\a})C
\]
is the straight line $\R x\sp\a$, where $x\sp\a \in G$ is the constant function
\[
x\sp\a \colon [0,2\pi] \to \R\sp2, \quad t \mapsto x\sp\a(t) = (x_1\sp\a(t),x_2\sp\a(t))=(\cos(\a/2),\sin(\a/2)),
\]
whose norm is one, $(x\sp\a)'(t)$ being identically zero.
Consequently, in the metric space $\Sigma$, the connected component $\Gamma$ containing the trivial solution
\[
p_* = (x_*,0,\l_*) =(x\sp{\tiny\pi/2},0,1) \in S \times \R \times \R
\]
of \eqref{abstract problem in E3} is diffeomorphic to a circle, as can be seen by means of the parametrization
\[
\a \in [0,4\pi] \mapsto (x\sp\a,\cos \a,\sin \a) \in S \times \R \times \R,
\]
which encounters four points of the set $\Sigma_0$ of the trivial solutions: two of them corresponding to the eigenvalue $\l_*=1$ of the unperturbed operator $L-\l C$ and the others corresponding to $\l\sp*=-1$.

We observe that the projection of $\Gamma$ onto the circle $\{(\e,\l) \in \R\sp2: \e\sp2+\l\sp2=1\}$ is a double covering map.

Let us check whether or not the trivial solution $p_* = (x_*,0,\l_*)$ is simple.
We have already shown that $L$ is Fredholm of index zero. Hence, so is the operator $A = L -\l_* C$, as a compact linear perturbation of $L$.
Moreover, $\Ker A = \R x_* = \R x\sp{\tiny\pi/2}$.
Therefore, $A$ satisfies the conditions (0) and (1) of Proposition \ref{unsolvability} and, consequently, $p_*$ is simple if and only if the equation $A x = C x_*$ is unsolvable.

Since $C x_* = Cx\sp{\tiny\pi/2}$ is the constant function $t \mapsto (\sqrt{2}/2,-\sqrt{2}/2)$, the trivial solution $p_*$ is simple provided that the problem
\begin{equation}
\label{non lo so}
\begin{cases}
x'_1(t) + x_1(t) - x_2(t) = \sqrt{2}/2,\\
x'_2(t) + x_1(t) - x_2(t) = -\sqrt{2}/2,\\
x_1(0)=x_1(2\pi), \; x_2(0)= x_2(2\pi).
\end{cases}
\end{equation}
is unsolvable; and this is clearly true since any solution $t \mapsto (x_1(t),x_2(t))$ of the first two equations cannot be periodic, the function $t \mapsto x_1(t) - x_2(t)$ being strictly increasing.
Analogously, one can check that all the other three trivial solutions of $\Gamma$ are simple.
Moreover, $\Gamma \setminus \Sigma_0$ is the disjoint union of four arcs, each of them satisfying the assertion of Theorem \ref{continuation}.

Notice that the solutions which are not in $\Gamma$ are all trivial and correspond to the eigenpairs $(0,\pm \sqrt{1+n\sp2})$, with $n \in \N$. For example, the solutions of the type $(x,0,\sqrt2)$ are the elements of the geometric circle
\[
\big(\Ker(L-\sqrt2 C) \cap S\big) \times \{(0,\sqrt2)\},
\]
which is a connected component of $\Sigma$ contained in $\Sigma_0$.
Consequently, according to Remark \ref{the component}, none of the solutions of this circle satisfies the assertion of Theorem \ref{continuation}. The fact that they are trivial but not simple shows that, in Theorem \ref{continuation}, the assumption that the solution $(x_*,0,\l_*)$ is simple cannot be removed.
\end{example} 

\medskip
We close with an elementary example showing that, in Conjecture \ref{conjecture-intro}, the assumption that $(x_*,0,\l_*)$ is simple cannot be removed.

\begin{example} 
\label{example4}
Let $G=H=\R\sp2$ and consider the linear problem
\begin{equation}
\label{problem in E4}
\left\{
\begin{array}{rcc}
-\e x_2 \eql \l x_1,\\
-2x_1 + \e x_1 \eql \l x_2,\\[0pt]
x_1\sp2+x_2\sp2 \eql 1.
\end{array}\right.
\end{equation}
The operators $L$ and $N$ are $(x_1,x_2) \mapsto (0,-2x_1)$ and $(x_1,x_2) \mapsto (-x_2,x_1)$, respectively, and $C=I$ is the identity.
The operator $L$ has a unique eigenvalue, $\l_* = 0$, whose multiplicity is two, and the kernel of $L - \l_* I = L$ is spanned by the unit eigenvector $x_* = (0,1)$.

The set $\mathcal E$ of the eigenpairs of \eqref{problem in E4} is the circle $\e(\e-2)+\l^2 = 0$.
Therefore, $\Sigma$ is necessarily bounded, and so is its connected component containing the trivial solution $p_* = (x_*,0,\l_*)$.
This component cannot meet any solution $p^* = (x^*,0,\l^*)$ with $\l^* \not= \l_*$, since our unperturbed problem has only one eigenvalue.

Thus, at least one of the following two possibilities holds true: the Conjecture \ref{conjecture-intro} is false or the trivial solution $p_*$ is not simple.
e are not sure about the conjecture, but of course $p_*$ is not simple, as one can easily check, for example, by means of Proposition \ref{unsolvability}.
\end{example} 



\end{document}